%% file: main.tex
\documentclass[12pt]{amsart}
\usepackage[utf8]{inputenc}
\usepackage[english,bidi=default]{babel}
\usepackage{lipsum}
\usepackage{hyperref}
\usepackage{amsmath}
\usepackage{amssymb}
\usepackage{mathtools}
\usepackage{bm}
\usepackage{physics}
\usepackage{afterpage}
\usepackage{fullpage}
\usepackage{dsfont}
\usepackage{amsthm}
	\theoremstyle{plain}% default
	\newtheorem{thm}{Theorem}[section]
	\newtheorem{lem}[thm]{Lemma}
	\newtheorem{prop}[thm]{Proposition}
	\newtheorem{cor}[thm]{Corollary}
	\theoremstyle{definition}
	\newtheorem{defn}[thm]{Definition}
	\newtheorem{ex}[thm]{Example}

\title{Subhomogeneous Operator Systems and Classification of Operator Systems Generated by $\Lambda$-Commuting Unitaries}
\author{Ran Kiri}

\include{macros}

\begin{document}
\maketitle
\begin{abstract}
A unital $C^*$-algebra is called $N$-subhomogeneous if its irreducible representations are finite dimensional with dimension at most $N$. We extend this notion to operator systems, replacing irreducible representations by boundary representations. This is done by considering the matrix state space associated with an operator system and identifying the boundary representations as absolute matrix extreme points. We show that two $N$-subhomogeneous operator systems are completely order equivalent if and only if they are $N$-order equivalent. Moreover, we show that a unital $N$-positive map into a finite dimensional $N$-subhomogeneous operator system is completely positive. We apply these tools to classify pairs of $q$-commuting unitaries up to $*$-isomorphism. Similar results are obtained for operator systems related to higher dimensional non-commutative tori.
\end{abstract}
\section{Introduction and preliminaries}
\subsection{Introduction}
Let $q,q'$ be two complex numbers of modulus $1$, and let $(u,v),(u',v')$ be two pairs of $q$-commuting unitaries and $q'$-commuting unitaries, respectively. If $\phi:C^*(u,v)\to C^*(u',v')$ is a $*$-homomorphism such that $\phi(u)=u',\phi(v)=v'$, we get that $q=q'$. Therefore, no $*$-homomorphism can map a $q$-commuting pair to a $q'$-commuting pair, unless of course $q=q'$.\par
Now, consider $\cS(u,v)$ and $\cS(u',v')$, the operator systems generated by such unitaries. One can show (see Theorem \ref{thm:UnitariesUEP} below) that a unital and completely positive map between $\cS(u,v)$ to $\cS(u',v')$ which also maps $u\mapsto u'$,$v\mapsto v'$, extends to a $*$-homomorphism from $C^*(u,v)$ to $C^*(u',v')$. It follows that $q=q'$ whenever such a map exists.\par 
This leads to a natural question. Does the existence of a unital positive map $u\mapsto u',v\mapsto v'$ already implies that $q=q'$? Or maybe being merely positive is not enough, but $n$-positive will do? More generally, we ask when does the $n$-order structure encode all of the information about the complete order structure of an operator system.\par 
One particular case for which we know the answer is the case of operator systems acting on finite dimensional spaces. When $\cS$ is acting on an $N_{\cS}$-dimensional space, Choi proved (see Theorem \ref{thm:ChoiPos}) that any $N_{\cS}$-positive map into $\cS$ is completely positive. In particular, if $\cS,\cR$ both act on a finite dimensional spaces of dimensions $N_{\cS}$ and $N_{\cR}$ respectively, we get that a $\max\{N_{\cS},N_{\cR}\}$-order isomorphism is a complete order isomorphism. There are, however, operator systems which cannot be faithfully represented on a finite dimensional space.\par 
In this paper we show that the answers to the above questions are related to the notion of subhomogeneous $C^*$-algebras, and to an extension of this notion to operator systems. Recall that a unital $C^*$-algebra is called \emph{$n$-subhomogeneous} if its irreducible representations are finite dimensional with dimension at most $n$, and \emph{subhomogeneous} if it is $n$-subhomogeneous for some $n\in\mathbb{N}$. We shall show in Theorem \ref{prop:Nsubhomogeneity}, that a $C^*$-algerba is subhomogeneous if its associated matrix state space, $\UCP{\cS}$, has matrix extreme points at level at most $n$. Inspired by this result we define an $n$-subhomogeneous operator system, replacing irreducible representations by finite dimensional boundary representations.\par 
Using this framework, we show in Theorem \ref{thm:NorderisComplete} that $n$-order isomorphisms between $N_{\cS}$-subhomogeneous and $N_{\cR}$-subhomogenous operator systems (with $N_{\cS},N_{\cR}\leq n$) is a complete order isomorphism, proving that for $N$-subhomogeneous operator systems with $N\leq n$, the $n$-order structure does encode the complete order structure. Moreover, we prove that a unital and $n$-positive map into a finite dimensional $n$-subhomogeneous operator system is completely positive, giving an extension to Choi's theorem.\par 
We then consider finite dimensional operator systems generated by unitaries. We shall prove in Theorem \ref{thm:equivalenceOfUnitaries} that two $d$-tuples of unitaries, generating $N_{\bs}$-subhomogeneous and $N_{\br}$-subhomogeneous operator systems (with $N_{\bs},N_{\br}\leq n$) are $n$-order equivalent if and only if they generate isomorphic $C^*$-algebras.\par 
In Section 5, we return to our motivating question regarding $q$-commuting unitaries. For $\cS=\cS(u,v)$, we provide a characterization of matrix extreme points in $\UCP{\cS}$ in terms of boundary representations for $\cS$, proving that when $q$ is of the form $e^{2\pi i\frac{k}{n}}$ for $(k,n)=1$, $\cS$ is $n$-subhomogeneous and when $q=e^{2\pi i\theta}$ with $\theta\in\qty[0,1]\notin\mathbb{Q}$, $\UCP{\cS}$ is not subhomogeneous (see Theorems \ref{thm:RationalMatrixRange} and \ref{thm:IrrationalMatrixRange}).\par 
As a consequence, we get that if $(u,v)$ is a pair of $q$-commuting unitaries and $(u',v')$ is a pair of $q'=e^{2\pi i\frac{k}{n}}$-commuting unitaries, then the existence of a unital $n$-positive map taking $(u,v)$ to $(u',v')$, implies $q=q'$, and the map is in fact completely positive (this is Theorem \ref{thm:qSeperation}). Similar results are obtained for higher dimensional non-commutative tori in Theorems \ref{thm:LambdaNposIsUCP} and \ref{thm:LambdaEquivalence}.\par 
The results we obtained show that in certain cases, the exitence of an $n$-order isomorphism implies the existence of a complete order isomorphism. It is natural to ask: how low can one take $n$ to be? In Section 6, we treat the question of whether the existence of a unital isometry (which is in particular an order isomorphism) implies a complete order equivalence. We provide a family of examples of $q=e^{2\pi i\frac{k}{n}}$-commuting unitaries which are isometrically isomorphic to a pair of $q=e^{2\pi i\frac{n-k}{n}}$-commuting unitaries (although they are not $n$-order isomorphic).
\subsection{Positive maps and order isomorphisms}
Given a unital $C^*$-algebra $\cA$, we say that $\cS$ is an \emph{operator system} if it contains the unit and it is closed under the $*$-operation. Given a tuple $\bs=\qty(s_1,\dots,s_d)$ of elements in $\cA$, we denote by $\cS(\bs)$ the operator system generated by $\bs$. There is a natural partial ordering endowed on $\cS$, in which $s\geq s'$ if $s-s'$ is positive in $\cA$, which we will call the order-structure or $1$-order structure on $\cS$. Identifying $\cA$ with some subalgebra of $B(\cH$) for some $\cH$, we can endow $\mat{n}(\cA)$ with the order structure from $\mat{n}(B(\cH))\cong B(\cH^n)$, where $\cH^n$ is the $n$-fold direct sum of $\cH$ with itself. This gives us a notion of positivity for elements in $\mat{n}\qty(\cA)$ and therefore for elements in $\mat{n}\qty(\cS)$. We will call the partial ordering endowed on $\mat{n}\qty(\cS)$ the $n$\emph{-order structure} of $\cS$. We shall loosely refer to the totality of $n$-order structures on $\cS$ as the \emph{complete order structure}.\par
For two operator systems $\cS,\cR$, and a linear map $\phi:\cS\to\cR$, we can define $\phi_n:\mat{n}\qty(\cS)\to\mat{n}\qty(\cR)$ for all $n\in\mathbb{N}$ by acting entry-wise, namely, $\phi_n\qty(s_{ij})=\qty(\phi(s_{ij}))$ for all $\qty(s_{ij})\in\mat{n}\qty(\cS)$. We say that $\phi$ is \emph{positive} if $\phi\qty(s)$ is positive in $\cR$ whenever $s$ is positive in $\cS$. We also say that $\phi$ is $n$\emph{-positive} if $\phi_n$ is a positive map, and we say that $\phi$ is completely positive if it is $n$-positive for all $n\in\mathbb{N}$. Note that $n$-positive maps are precisely the maps that preserve that $n$-order structure of $\cS$.\par 

Completely positive maps arise naturally in the study of operator systems through the $C^*$-algebras that they generate. For example, given an operator system $\cS$ contained in some $C^*$-algebra $\cA$, and a $*$-homomorphism $\pi:\cA\to\cB$ for some $C^*$-algebra $\cB$, we know that the restriction $\left.\pi\right|_{\cS}$ is a completely positive map. If we assume that $\cS\subset B(\cH),\cR\subset B(\cK)$ for some Hilbert spaces $\cH,\cK$, any operator $V:\cK\to\cH$ induces a completely positive map $\phi_V:\cS\to B(\cK)$ given by $\phi(s)=V^*sV$. In fact, for an operator system $\cS$, an application of \cite[Theorem 1.2.3]{Arveson69} and \cite[Theorem 1]{Stinespring} shows that any completely positive map $\phi:\cS\to B(\cH)$ is of the form $\phi=V^*\left.\pi\right|_{\cS}(\cdot)V$ where $\pi:C^*(\cS)\to B(\cH')$ is a $*$-homomorphism for some Hilbert space $\cH'$ and $V:\cH\to\cH'$ is a bounded operator.\par 
For $n$-positive maps which are not completely positive, we do not have such a characterization, and it might seem that these maps ``miss'' some of the $C^*$-algebraic structure encoded in $\cS$ about $C^*(\cS)$.
\begin{defn}
Let $\cS,\cR$ be two operator systems and let $\phi:\cS\to\cR$ be a unital linear map. We say that:
\begin{itemize}
\item $\phi$ is an \emph{order isomorphism} if $\phi$ is a bijection and both $\phi,\phi^{-1}$ are positive.
\item $\phi$ is an $n$\emph{-order isomorphism} if $\phi$ is a bijection and both $\phi,\phi^{-1}$ are $n$-positive.
\item $\phi$ is a \emph{complete order isomorphism} if $\phi$ is an $n$-order isomorphism for all $n$.
\end{itemize}
We say that $\cS,\cR$ are \emph{order/$n$-order/completely order isomorphic} if there exists an order/$n$-order/complete order isomorphism between them.
\end{defn}
\begin{defn}
Let $\bs,\br$ be $d$-tuples generating operator systems $\cS,\cR$ respectively. We say that:
\begin{itemize}
\item $\bs,\br$ are \emph{order/$n$-order/completely order equivalent} if there exists an order/$n$-order/complete order isomorphism between $\cS,\cR$ mapping $s_i$ to $r_i$ for all $i$.
\item $\bs,\br$ are \emph{$*$-isomorphic} if there exists a $*$-isomorphism of $C^*(\cS),C^*(\cR)$ mapping $s_i$ to $r_i$ for all $i$.
\end{itemize}
\end{defn}
It is well known that for all $n\in\mathbb{N}$, there are $n$-positive maps which are not completely positive. Indeed, if $N>n$, the map $\phi^{(n)}:\mat{N}\to\mat{N}$ given by $\phi^{(n)}(A)=n{\rm tr}\qty(A)I_N-A$ is an $n$-positive map but not completely positive (see \cite{tomiyama1985geometry} for proof).\par
Note that if $\pi:C^*(\cS)\to \cA$ is a $*$-homomorphism, then $\left.\pi\right|_{\cS}$ is a completely positive map. This means that $\left.\pi\right|_{\cS}$ can be extended to the generated algebra, to a completely positive map which is also multiplicative. Although any completely positive map can be extended from an operator system to the $C^*$-algebra that it generates, there might not be any extension which is multiplicative.\par
In this work we study to what extent does the $n$-order structure determine the complete order structure (and perhaps, the $C^*$-algebraic structure). The following theorem shows a case in which the $n$-order structure does determine the complete order structure.
\begin{thm}[Theorem 5, \cite{choi1972positive}]
\label{thm:ChoiPos}
Let $\cA$ be a unital $C^*$-algebra, and let $\phi:\cA\to\mat{n}$ be $n$-positive. Then, $\phi$ is completely positive.
\end{thm}
A closer look at the proof shows that this remains true if we replace $\cA$ by an operator system $\cS$. Therefore, if $\cS\subset\mat{n}$ and $\cR\subset\mat{m}$ are two operator systems acting on finite dimensional spaces, then begin $\max\{n,m\}$-order isomorphic is equivalent to being completely order isomorphic. This is a case where the $N$-order structure for some (maybe large) $N$, encodes the complete order structure. In Section 3 we find a generalization of this phenomenon.
\subsection{Matrix convex sets and matrix extreme points}
\begin{defn}
Let $V$ be a vector space. A \emph{matrix convex set} over $V$ is a collection $K=\qty(K_n)_{n\in\mathbb{N}}$ of subsets $K_n\subset\mat{n}\qty(V)$ such that:
\begin{equation}
	\label{eq1ch1}
	\sum\limits_{i=1}^k\gamma_i^*v_i\gamma_i\in K_n
\end{equation}
for all $v_i\in K_{n_i}$ and $\gamma_i\in\mat{n_i}{n}$ for $i=1,\dots,k$ satisfying $\sum_{i=1}^k\gamma_i^*\gamma_i=\mathds{1}_n$.
\end{defn}
The sum in (\ref{eq1ch1}) is called a \emph{matrix convex combination}. When $\gamma_i$ is surjective for all $i$ (in particular, $n_i\leq n$) we say that this matrix convex combination is \emph{proper}.\par
When $V$ is a locally convex vector space, we can also say that $K=\qty(K_n)_{n\in\mathbb{N}}$ is a \emph{compact matrix convex set} if each $K_n$ is compact in $\mat{n}\qty(V)$ in the product topology.\par 
The following are the two prime examples of compact matrix convex sets, the first associated with an operator system and the second is associated with a $d$-tuple of operators.
\begin{defn}
Let $\cS$ be an operator system. The \emph{matrix state space} of $\cS$ is the collection $\UCP{\cS}=\qty(\UCP{\cS}{\mat{n}})_{n\in\mathbb{N}}$, where:
\[
	\UCP{\cS}{\mat{n}}=\qty{\phi:\cS\to\mat{n}\middle|\phi\;\mbox{is UCP}}.
\]
\end{defn}
Note that $\UCP{\cS}$ is a matrix convex set over $\cS^*$, and it is compact (see the proof of Lemma 1.2.4 in \cite{Arveson69}).
\begin{defn}
Let $\cA$ be a unital $C^*$-algebra, and let $\bs=\qty(s_1,\dots,s_d)\in\cA^d$ be a $d$-tuple. The \emph{matrix range} of $\bs$ is the collection $\W{\bs}=\qty(\W{n}{\bs})_{n\in\mathbb{N}}$, where:
\[
	\W{n}{\bs}=\qty{\qty(\phi(s_1),\dots,\phi(s_d))\middle|\phi\in\UCP{\cS}{\mat{n}}}
\]
\end{defn}
This a matrix convex set over $\mathbb{C}^d$. Matrix ranges are the typical form of compact matrix convex sets over $\mathbb{C}$ (see \cite[Proposition 31]{loebl1981some}). The following theorem shows that matrix ranges are related to $n$-order and complete order equivalence of $d$-tuples. Since this theorem will be used extensively throughout this work, we present it here as well.
\begin{thm}[Theorem 5.1, \cite{DDSS}]
\label{thm:DDSS}
Let $\bs\in B(\cH)^d,\br\in B(\cH)^d$ be $d$-tuples of operators.
\begin{enumerate}
\item Given $n\in\mathbb{N}$, if there exists a unital $n$-positive map $\phi:\cS(\bs)\to\cS(\br)$ sending $\bs$ to $\br$, then $\W{n}{\br}\subset \W{n}{\bs}$.
\item There exists a UCP map $\phi:\cS(\bs)\to\cS(\br)$ sending $\bs$ to $\br$ if and only if $\W{\br}\subset\W{\bs}$. If $\bs$ is a commuting tuples of normal operators, this inclusion is equivalent to $\sigma\qty(\br)\subset \W{1}{\bs}$, which is also equivalent to $\W{1}{\br}\subset\W{1}{\bs}$.
\item There exists a unital and completely isometric map $\phi:\cS(\bs)\to\cS(\br)$ sending $\bs$ to $\br$ if and only if $\W{\bs}=\W{\br}$.
\end{enumerate}
\end{thm}
\begin{defn}
Let $K=\qty(K_n)_{n=1}^{\infty}$ and $L=\qty(L_n)_{n=1}^{\infty}$ be two matrix convex sets over vector spaces $V$ and $W$, respectively. A \emph{matrix affine map} from $K$ to $L$ is a collection $\bm{\theta}=\qty(\theta_n)_{n=1}^{\infty}$ of mappings $\theta_n:K_n\to L_n$ such that:
\[
	\theta_n\qty(\sum_{i=1}^k\gamma_i^*v_i\gamma_i)=\sum\limits_{i=1}^k\gamma_i^*\theta_{n_i}\qty(v_i)\gamma_i
\]
for all $v_i\in K_{n_i}$ and $\gamma_i\in\mat{n_i}{n}$ for $i=1,\dots,k$ satisfying $\sum_{i=1}^k\gamma_i^*\gamma_i=\mathds{1}_n$. When $V,W$ are both topological vector spaces, we say that $\bm{\theta}$ is continuous if $\theta_n$ is continuous for all $n\in\mathbb{N}$. Such a map is called a \emph{matrix affine homeomorphism} if each $\theta_n$ is a homeomorphism.
\end{defn}
\begin{ex}
Let $\bs$ be a $d$-tuple and set $\cS=\cS\qty(\bs)$ to be the operator system they generate. The collection $\bm{\theta}=\qty(\theta_n)_{n=1}^{\infty}$ of mappings $\theta_n:\UCP{\cS}{\mat{n}}\to\W{n}{\bs}$ defined by $\theta_n\qty(\phi)=\qty(\phi(s_1),\dots,\phi(s_d))$ is a matrix affine homeomorphism. Indeed, it is surjective by the definition of $\W{\bs}$, and it is injective because positive maps are self-adjoint, which means that any $\phi\in\UCP{\cS}{\mat{n}}$ is uniquely defined by the values of $\phi(s_i)$ for all $i$. Finally, $\theta_n$ is continuous because the topology induced on $\UCP{\cS}{\mat{n}}$ is the weak* topology, which means that $\phi_m$ converges to $\phi$ if and only if $\phi_m(s)$ converges in $\mat{n}$ to $\phi(s)$ for all $s\in \cS$, which in particular means that $\phi_m(s_i)$ converges to $\phi(s_i)$ for all $i$ (and therefore $\theta_n(\phi_m)$ converges to $\theta_n(\phi)$). $\theta_n$ is a continuous bijection from a compact space to a Hausdorff space, so $\theta_n^{-1}$ is also continuous.\par 
This shows that we can switch between the two perspectives when studying operator systems generated by $d$-tuples of operators.
\end{ex}
Given a compact matrix convex set $K=\qty(K_n)_{n=1}^{\infty}$, $A(K)$ denotes the set of all functions $F=\qty(F_n)_{n=1}^{\infty}$ with $F_n:\mat{n}\qty(K)\to\mat{n}$ such that $F_1$ is continuous, and $F$ is matrix affine, in the sense that:
\[
	F_n\qty(\sum\limits_{i=1}^k\gamma_i^*v_i\gamma_i)=\sum\limits_{i=1}^k\gamma_i^*F_{n_i}(v_i)\gamma_i,
\]
for every $v_i\in K_{n_i},\gamma_i\in\mat{n_i}{n}$ such that $\sum_{i=1}^k\gamma_i^*\gamma_i=\mathds{1}_n$. In the discussion leading to \cite[Proposition 3.5]{WebWink}, it was shown that $A(K)$ can be endowed with a structure of an abstract operator system. The following proposition shows us a deep connection between matrix state spaces of operator systems and their complete order structure.
\begin{prop}[Proposition 3.5, \cite{WebWink}]
\label{prop:WebWink35}
\begin{enumerate}
\item If $\cR$ is an operator system, then $\UCP{\cR}$ is a self-adjoint compact matrix convex set in $\cR^*$, equipped with the weak* topology, and $A(\UCP{\cR})$ and $\cR$ are isomorphic as operator systems.
\item If $K=\qty(K_n)_{n=1}^{\infty}$ is a compact matrix convex set in a locally convex space $V$, then $A\qty(K)$ is an operator system, and $K$ and $\UCP{A(K)}$ are matrix affinely homeomorphic as operator systems.
\end{enumerate}
\end{prop}
We will also make use of the particular isomorphism which appeared in the proposition. The set $A(K)$ has a positive cone in which $F$ is positive if $F_n\qty(v)$ is positive for all $n\in\mathbb{N},v\in K_n$. In the case of $A(\UCP{\cS})$ and $\cS$, The complete order isomorphism of $\cS$ and $A(\UCP{\cS})$, is implemented by the mapping $s\mapsto\delta_{\cS}=\qty(\qty(\delta_{\cS})_n)_{n=1}^{\infty}$, where:
\[
	\qty(\delta_{\cS})_n\qty(\phi)=\phi\qty(s)
\]
for all $\phi\in\UCP{\cS}{\mat{n}}$. This gives us a concrete way of studying $\cS$ through its matrix state space.\par
When we shift our perspective from an operator system to its matrix state space, we can use the theory of matrix convex sets to study it. In particular, in Section 2, we use special points in $\UCP{\cS}$ which are the \emph{matrix extreme points} to study the structure of $\UCP{\cS}$.
\begin{defn}
Let $K=\qty(K_n)_{n=1}^{\infty}$ be a matrix convex set and let $v\in K_n$ for some $n\in\mathbb{N}$. We say that $v$ is a \emph{matrix extreme point} in $K$, if whenever $v$ is a proper matrix convex combination:
\[
	v=\sum\limits_{i=1}^k\gamma_i^*v_i\gamma_i
\]
then for all $i$, $n_i=n$ and $v=u_i^*v_iu_i$ for some unitary $u_i\in\mat{n}$.
\end{defn}
We set $\partial K=\qty(\partial K_n)_{n=1}^{\infty}$ to denote the set of matrix extreme points in $K_n$ for all $n\in\mathbb{N}$.\par Webster and Winkler proved in \cite[Theorem 4.3]{WebWink} an analogue of the Krein-Milman theorem for matrix extreme points in a compact matrix convex set, namely that any compact matrix convex set is the closed matrix convex hull of its matrix extreme points. In that sense, matrix extreme points are analogous to extreme points in locally convex vector spaces. There is another, stronger notion of matrix extreme points introduced in \cite{evert2018extreme}, which are called absolute matrix extreme points.
\begin{defn}
Let $K=\qty(K_n)_{n=1}^{\infty}$ be matrix convex and let $v\in K_n$ for some $n\in\mathbb{N}$. We say that $v$ is an \emph{absolute matrix extreme point} in $K$ if whenever $v$ is a matrix convex combination (not necessarily proper):
\[
	v=\sum\limits_{i=1}^k\gamma_i^*v_i\gamma_i
\]
such that $\gamma_i$ are all non-zero, then for all $i$:
\begin{itemize}
\item if $n_i\leq n$, then $n_i=n$ and $v=u_i^*v_iu_i$ for some unitary $u_i\in\mat{n}$.
\item if $n_i>n$, then there exists some $w_i\in K_{n_i-n}$ such that $v_i=u_i^*(v\oplus w_i)u_i$ for some unitary $u_i\in K_{n_i}$.
\end{itemize}
The set of absolute matrix extreme points in $K$ is denoted by ${\rm Abex}(K)$.
\end{defn}
In other words, matrix extreme points are points which cannot be written as matrix combinations from below, except in a trivial way, and absolute matrix extreme points are points which cannot be written as matrix combination from above or below, except in a trivial way. It is obvious that an absolute matrix extreme point is in particular matrix extreme.\par 
Unlike matrix extreme points, absolute matrix extreme points do not give a nice generalization of the Krein-Milman theorem for any compact matrix convex set. In fact, there are compact matrix convex sets which do not have any absolute matrix extreme points (see \cite[Example 6.30]{Kriel}, and also Theorem \ref{thm:IrrationalMatrixRange}). But as we will see in Section 2, when considering $\UCP{\cS}$ for some operator system $\cS$, the matrix extreme points are pure UCP maps, and absolute matrix extreme points are restrictions of boundary representations for $\cS$, which makes it easier to identify them.
\section{Matrix extreme points in compact matrix convex sets}
In this section we prove the following theorem:
\begin{thm}
\label{thm:absoluteIsBoundary}
Let $\cS$ be an operator system and let $\cA=C^*(\cS)$. A point $\phi\in\UCP{\cS}{\mat{n}}$ is an absolute matrix extreme points of $\UCP{\cS}$ if and only if it extends to a boundary representation on $\cA$.
\end{thm}
Several proofs for particular cases of this result appear in the literature. In \cite[Corollary 6.28]{Kriel}, there is a proof of this theorem for operator systems generated by self-adjoint matrices, and in \cite[Theorem 4.2]{kleski2014boundary}, there is a version of this theorem for general compact matrix convex sets $K$ such that $A(K)$ acts on a finite dimensional space. Another matricial version appears in \cite[Theorem 3.10]{evert2018extreme}.\par 
Before getting to the proof we will need some terminology. Recall that a UCP map $\phi:\cS\to B(\cH)$ is called \emph{pure} if whenever $\psi$ is a CP map such that $\phi-\psi$ is CP, then $\psi=t\phi$ for $t\in[0,1]$. Given some $s\in\cS,h\in\cH$, we say that a UCP map $\phi:\cS\to B(\cH)$ is \emph{maximal} at $(s,h)$ if whenever $\phi(\cdot)=V^*\psi(\cdot)V$ for an isometry $V:\cH\to \cK$ and $\psi\in\UCP{\cS}{B(\cK)}$, we get that $\norm{\phi(s)h}=\norm{\psi(s)Vh}$. A UCP map which is maximal at every $(s,h)\in \cS\times\cH$ is called \emph{maximal}. Note that $\phi$ is maximal if and only if whenever $\phi(\cdot)=V^*\psi(\cdot)V$ for some isometry $V:\cH\to\cK$, then $\psi(\cdot)=V\phi(\cdot)V^*\oplus\rho(\cdot)$ for some UCP $\rho:\cS\to B((V\cH)^{\perp})$.\par 
A UCP map $\phi:\cS\to B(\cH)$ is said to have the \emph{unique extension property} if it has a unique UCP extension to some $\Phi:C^*(\cS)\to B(\cH)$, and $\Phi$ is multiplicative. Lastly, we say that a $*$-representation $\pi:C^*(\cS)\to B(\cH)$ is a \emph{boundary representation} for $\cS$ if it is irreducible, and $\left.\pi\right|_{\cS}$ has the unique extension property.\par 
We will now state some known results which will be needed for our proof.
\begin{thm}[Theorem B, \cite{farenick2000extremal}]
\label{thm:matExtremeArePure}
Let $\cS$ be an operator system in a unital $C^*$-algebra $\cA$. Then:
\begin{enumerate}
\item A matrix state $\phi$ on $\cS$ is a matrix extreme point of $\UCP{\cS}$ if and only if $\phi$ is pure.
\item Every pure UCP map $\phi\in\UCP{\cS}$ extends to a pure map $\Phi\in\UCP{\cA}$.
\end{enumerate}
\end{thm}
The following proposition is due to Farenick and Tessier, showing that restrictions of boundary representations are always pure and thus matrix extreme:
\begin{prop}[Proposition 2.12, \cite{farenick2022purity}]
\label{prop:boundaryRepsArePure}
Let $\cA$ be a unital $C^*$-algebra generated by an operator system $\cS$ and let $\pi:\cA\to B(\cK)$ be a boundary representation for $\cS$. Then, $\left.\pi\right|_{\cS}$ is pure.
\end{prop}
The next step is to show that restrictions of boundary representations are maximal. This observation was made in the context of operator algebras by Dritschel and McCullough in \cite[Theorem 1.1]{dritschel2005boundary} (who followed Muhly and Solel, see \cite[Theorem 1.2]{muhly1998algebraic}). The following reformulation in the context of operator systems is due to Arveson.
\begin{prop}[Proposition 2.4, \cite{arveson2008noncommutative}]
\label{prop:UEPiffMaximal}
Let $\cS$ be an operator system and let $\cA=C^*(\cS)$. Then, a map $\phi:\cS\to B(\cH)$ is maximal if and only if it has the unique extension property.
\end{prop}
We can now proceed to the proof of the main theorem of this section.
\begin{proof}[Proof of Theorem \ref{thm:absoluteIsBoundary}]
Assume that $\phi$ extends to a boundary representation of $\cA$, and assume that
\[
	\phi=\sum\limits_i\gamma_i^*\phi_i\gamma_i
\]
is a matrix convex combination with $\gamma_i$ non-zero for all $i$. For each $i$, we get that $\phi-\gamma_i^*\phi_i\gamma_i$ is completely positive.  By Proposition \ref{prop:boundaryRepsArePure}, $\phi$ is pure, which means that there exists some $t\in(0,1]$ such that $\gamma_i^*\phi_i\gamma_i=t\phi$, from which it holds that
\[
	\qty(t^{-\frac{1}{2}}\gamma_i)^*\phi_i\qty(t^{-\frac{1}{2}}\gamma_i)=\phi.
\]
$\phi$ and $\phi_i$ both being UCP, we get that $\qty(t^{-\frac{1}{2}}\gamma_i)^*\qty(t^{-\frac{1}{2}}\gamma_i)=\mathds{1}_n$, so that whenever $n_i\leq n$, we get that $n_i=n$ and that this matrix is unitary. When $n_i>n$, we get that $\phi_i$ is a dilation of $\phi$, and since $\phi$ has the UEP (being a boundary representation), Proposition \ref{prop:UEPiffMaximal} implies that $\phi_i$ is unitarily equivalent to $\phi\oplus\rho_i$ for some $\rho_i\in\UCP{\cS}{\mat{n_i-n}}$. Therefore, $\phi$ is an absolute matrix extreme point.\par
As for the converse, note that if $\phi$ is an absolute matrix extreme point, it is pure by Theorem \ref{thm:matExtremeArePure}.  By \cite[Lemma 2.1]{davidson2015choquet}, we can finish the proof by showing that $\phi$ is maximal. Assuming it is not maximal, we apply \cite[Lemma 2.3]{davidson2015choquet}, and find some pure UCP map $\psi\in\UCP{\cS}{\mat{n+1}}$ such that $\phi=V^*\psi V$.  Since that is a matrix convex combination and $\phi$ is an absolute matrix extreme point, we get that $\psi=U^*(\phi\oplus\rho)U$ for some unitary $U$ and a state $\rho\in\UCP{\cS}{\mathbb{C}}$. But now, defining:
\[
\gamma_1=\left(\begin{array}{ccccc}1 & 0  & \dots & 0 & 0 \\0 & 1 & \dots & 0 & 0 \\ \vdots &\ddots &\ddots &\vdots & \vdots \\ 0 & \dots & 0 & 1 & 0 \end{array}\right)\in\mat{n}{n+1},\quad \gamma_2=\qty(0,\dots,0,1)\in\mat{1}{n+1},
\]
we get that:
\[
	\psi=U^*\qty(\gamma_1^*\phi\gamma_1+\gamma_2^*\rho\gamma_2)U=\qty(\gamma_1U)^*\phi\qty(\gamma_1U)+\qty(\gamma_2U)^*\rho\qty(\gamma_2U)
\]
is a non-trivial proper matrix convex combination, which shows $\psi$ is not matrix extreme (and thus not pure, by Theorem \ref{thm:matExtremeArePure}). This proves that $\phi$ is maximal. Finally, it follows that $\phi$ is pure and maximal and thus extends to a boundary representation for $\cS$.
\end{proof}
It follows from Theorem \ref{thm:absoluteIsBoundary} that the restrictions to $\cS$ of finite dimensional boundary representations are precisely the absolute matrix extreme points of $\UCP{\cS}$. We would also like to find out whether $\UCP{\cS}$ contains information about infinite dimensional boundary representations, and to characterize matrix extreme points which are not absolute. The first part was answered by Davidson and Kennedy:
\begin{thm}[Theorem 2.4, \cite{davidson2015choquet}]
\label{thm:pureDilatesToBoundary}
Let $\cS$ be an operator system and let $\phi:\cS\to B(\cH)$ be a pure UCP map. Then, $\phi$ dilates to a boundary representation for $C^*(\cS)$.
\end{thm}
This means that any matrix extreme point which is not absolute (namely, it is not already a restriction of a boundary representation) is of the form $V^*\left.\pi\right|_{\cS}V$ for some boundary representation $\pi$. In Section 5 we will see an example of an operator system for which $\UCP{\cS}$ has a matrix extreme point of this form, with $\pi$ being infinite dimensional.\par
As for the second part, we saw that matrix extreme points which are not absolute, dilate to another matrix extreme point which lies in the next level of $\UCP{\cS}$. This gives a dilation-theoretic characterization of these points, namely, matrix extreme points are either restrictions of boundary representations or they dilate non trivially to another matrix extreme point at the next level. In the last case, they are also compressions of restrictions of boundary representations.\par 
The above were stated for matrix state spaces, but through \ref{prop:WebWink35}, we have the following corollary for a general compact matrix convex set:
\begin{cor}
\label{cor:MatExtChar}
Let $K=\qty(K_n)_{n=1}^{\infty}$ be a compact matrix convex set over a locally convex vector space $V$. Then $v\in\partial K_n$ is a matrix extreme point which is not absolute if and only if it is of the form $\gamma^*\tilde{v}\gamma$ for some $\tilde{v}\in\partial K_{n+1}$.
\end{cor}
\section{Subhomogeneous matrix convex sets}
In this section we introduce the notion of Subhomogeneous matrix convex sets. Recall that a unital $C^*$-algebra $\cA$ is called Subhomogeneous, if there exists some $N\in\mathbb{N}$ such that every irreducible representation of $\cA$ is of dimension less than or equal to $N$. We say that $\cA$ is $N$-Subhomogeneous, if $N$ is the smallest integer satisfying this condition.\par 
\begin{prop}
\label{prop:Nsubhomogeneity}
Let $\cA$ be a unital $C^*$-algebra. The following are equivalent:
\begin{itemize}
\item $\cA$ is $N$-Subhomogeneous.
\item $\partial\UCP{\cA}\subset\bigcup_{n=1}^N\UCP{\cA}{\mat{n}}$.
\end{itemize}
\end{prop}
\begin{proof}
Assume that $\cA$ is $N$-Subhomogeneous and let $\varphi\in\partial\UCP{\cA}{\mat{m}}$ for some $m$. Our goal is to show that $m\leq N$. Using \cite[Corollary 1.4.3]{Arveson69}, we know that $\varphi(x)=V^*\pi(x)V$ for some irreducible representation $\pi$ on some Hilbert space $\cK$, and $V:\cH\to \cK$ is an isometry. Since $\pi$ is an irreducible representation, it must be finite dimensional with dimension less then or equal to $N$. But then, we have that $\varphi$ is a compression of $\pi$, which means that $m\leq N$ and we are done.\par 
For the converse, assume that $\partial\UCP{\cA}\subset\bigcup_{n=1}^N\UCP{\cA}{\mat{n}}$, and let $\pi$ be an irreducible representation of $\cA$. If $\pi$ is finite dimensional, it is a matrix extreme point, which means it has to lie in some $\UCP{\cA}{\mat{n}}$ for $n\leq N$, and we are done. Otherwise, assume that $\pi:\cA\to B(\cH)$ with $\cH$ being infinite dimensional. Choose some $N+1$ dimensional subspace $\mathcal{M}\subset\mathcal{H}$, and define $V:\mathcal{M}\to\mathcal{H}$ by $Vm=m$. Then, we have that $\varphi(x)=V^*\pi(x)V$ is a matrix extreme point by \cite[Corollary 1.4.3]{Arveson69} which lies in $\UCP{\cA}{\mat{N+1}}$ which is a contradiction. This shows that all irreducible representations are finite dimensional and of dimension less then or equal to $N$, so $ \cA$ is $N$-Subhomogeneous.
\end{proof}
\begin{defn}
Let $K=\qty(K_n)_{n=1}^{\infty}$ be a compact matrix convex set over a locally convex vector space $V$. We say that $K$ is \emph{Subhomogeneous} if there exists some $N\in\mathbb{N}$ such that $\partial K_n=\emptyset$ for all $n>N$. If $N$ is the smallest integer satisfying these conditions, we say that $K$ is \emph{$N$-Subhomogeneous}.
\end{defn}
Proposition \ref{prop:Nsubhomogeneity} shows that if $K=\UCP{\cA}$ for some unital $C^*$-algebra $\cA$, $K$ is $N$-Subhomogeneous if and only if $\cA$ is Subhomogeneous.\par 
We use this relation in order to extend the notion of subhomogeneity to operator systems.
\begin{defn}
Let $\cS$ be an operator system. We say that $\cS$ is \emph{Subhomogeneous} if $\UCP{\cS}$ is Subhomogeneous, and that it is $N$-\emph{Subhomogeneous} if $\UCP{\cS}$ is $N$-Subhomogeneous.
\end{defn}
By Theorem \ref{thm:absoluteIsBoundary}, we see that $\cS$ is $N$-Subhomogeneous if and only if the finite dimensional boundary representatios for $\cS$ are of dimension less than or equal to $N$. We do not know whether a Subhomogeneous operator system may admit an infinite dimensional boundary representation. Note, however, that regardless of the answer to that question, subhomogeneous operator systems are completely normed by their finite dimensional boundary representations.\par 
In the rest of the section we show that Subhomogeneous compact matrix convex sets can be recovered from a finite number of levels.
\begin{prop}
\label{prop:NKreinMilman}
Let $K=\qty(K_n)_{n=1}^{\infty}$ be an $N$-Subhomogeneous compact matrix convex set over some locally convex vector space $V$. Then
\[
		K=\overline{{\rm co}}\qty(K_N).
\]
\end{prop}
\begin{proof}
Let $v$ be a matrix extreme point in $\partial K_n$ for some $n<N$. Then for any $u\in K_{N-n}$ we have that $v\oplus u\in K_N$ by matrix convexity. Therefore, any matrix convex combination of matrix extreme points, can be written as a matrix convex combination of points which all lie in $K_N$. Therefore, by applying the Webster-Winkler Krein-Milman theorem (\cite[Theorem 4.3]{WebWink}):
\[
	K=\overline{{\rm co}}\qty(\partial K)=\overline{{\rm co}}\qty(K_N)
\]
\end{proof}
In particular, we get that if $K=\qty(K_n)_{n=1}^{\infty}$ is the smallest compact matrix convex set for which the $N$-th level is $K_N$. For the special case of matrix convex sets over $\mathbb{C}^d$ for some $d\in\mathbb{N}$, we can improve this result.
For a compact matrix convex set $K=\qty(K_n)_{n=1}^{\infty}$ in $\mathbb{C}^d$ and some $N\in\mathbb{N}$, the set $\cW^{N\mbox{-min}}\qty(K)$ is defined as the smallest matrix convex set which has $K_N$ at level $N$. We can also describe those sets in terms of their matrix extreme points.
\begin{prop}
\label{prop:NgeneratedisMinimal}
Let $K=\qty(K_n)_{n=1}^{\infty}$ be a compact matrix convex set in $\mathbb{C}^d$. Then, $K=\cW^{N\mbox{-min}}\qty(K)$ if and only if it is $n$-Subhomogeneous for some $n\leq N$.
\end{prop}
\begin{proof}
Assume first that $K=\cW^{N\mbox{-min}}\qty(K)={\rm co}\qty(K_N)$. In that case, any point $v\in K_m$ for $m>N$ is a proper matrix convex combination of points from level $N$, which means it is not matrix extreme. Therefore, the highest level which contains a matrix extreme point is $N$ or lower, which proves the claim.\par 
For the converse, note that by \cite[Corollary 2.5]{hartz2021dilation} and Proposition \ref{prop:NKreinMilman}:
\[
	K=\overline{{\rm co}}\qty(K_n)={\rm co}\qty(K_n)\subset{\rm co}\qty(K_N),
\]
and by matrix convexity, we get equality. We finish the proof by noting that ${\rm co}\qty(K_N)$ is (by definition) the minimal matrix convex set which has $K_N$ at level $N$.
\end{proof}
\section{Subhomogeneous operator systems}
According to Proposition \ref{prop:WebWink35}, $\cS$ is completely order isomorphic to $A(\UCP{\cS})$, via the map $\delta_{\cS}:\cS\to A(\UCP{\cS})$ such that $\delta_{\cS}(s)=((\delta_{\cS}(s))_n)_{n\in\mathbb{N}}$ is given by:
\[
	\qty(\delta_{\cS}(s))_n\qty(\phi)=\phi(s)
\]
for all $n\in\mathbb{N}$ and $\phi\in\UCP{\cS}{\mat{n}}$. In this section we show that this map gives us a way to lift an equivalence of two matrix state spaces to an equivalence of the associated operator systems. 
\begin{prop}
\label{prop:mataffhomImpliesCOI}
Let $\cS$ and $\cR$ be two operator systems, and assume that $T=\qty(T_n)_{n=1}^{\infty}$ is a matrix affine homeomorphism of $\UCP{\cS}$ and $\UCP{\cR}$. Then, $T$ induces a complete order isomorphism of $\cS,\cR$.
\end{prop}
\begin{proof}
Define a mapping $T_*:A(\UCP{\cR})\to A(\UCP{\cS})$ by $T_*F=\qty(F_n\circ T_n)_{n=1}^{\infty}$. We claim that this is a complete order isomorphism of the two operator systems. In order to do so, we use the identification $M_n\qty(A(\UCP{\cR}))$ with $A(\UCP{\cR},\mat{n})$ and similarly for $\cS$. In that case, one may write $\qty(T_*)_nF$ as the map $\qty(F_m\circ T_m)_{m=1}^{\infty}$ for all $F\in A(\UCP{\cR},\mat{n})$. By definition, $\qty(T_*)_nF$ is positive if and only if $F_m\circ T_m(\phi)$ is positive for all $m\in\mathbb{N}$ and $\phi\in\UCP{\cS}{\mat{m}}$. But because $T$ is a bijection, we get this this is true if and only if $F_m\qty(\psi)$ is positive for all $m\in\mathbb{N}$ and $\psi\in\UCP{\cR}{\mat{m}}$, which is equivalent to $F$ being positive. $T_*$ is obviously unital (the unit does not depend on the argument), meaning the map is indeed UCP. Similar arguments for $T^{-1}$ shows that this map is a complete order isomorphism. Finally, the map $\delta_{\cS}^{-1}\circ T_*\circ \delta_{\cR}$ is a complete order isomorphism of $\cS,\cR$ as a composition of such.
\end{proof}
We are now ready to show that the $N$-order structure (for some large enough $N$) encodes the complete order structure of the operator system.
\begin{thm}
\label{thm:NorderisComplete}
Let $\cS,\cR$ be $N_{\cS}$-Subhomogeneous and $N_{\cR}$-Subhomogeneous operator systems, and set $N=\max\{N_{\cS},N_{\cR}\}$. Then, $\cS$ and $\cR$ are completely order isomorphic if and only if they are $N$-order isomorphic.
\end{thm}
Before proving this theorem, we need the next two propositions.
\begin{prop}
\label{prop:bijectionofMatExt}
Let $\cS,\cR$ be an $N$-order isomorphism and let $k\leq N$. Then, the collection $\bm{\phi}^*=\qty(\phi_n^*)_{n\in\mathbb{N}}$ of mappings $\phi_n^*:L\qty(\cR,\mat{n})\to L(\cS,\mat{n})$ given by $\phi_n^*(\psi):=\psi\circ\phi$ restricts to a bijection of $\UCP{\cS}{\mat{k}}$ and $\UCP{\cR}{\mat{k}}$ for all $k\leq N$. In addition, the restriction is a bijection of $\partial\UCP{\cS}{\mat{k}}$ and $\partial\UCP{\cR}{\mat{k}}$ for $k\leq N$.
\end{prop}
\begin{proof}
First note that $\bm{\phi}^*$ is a matrix affine map. Indeed, by definition we get that:
\[
	\phi_n^*\qty(\sum\limits_{i=1}^k\gamma_i^*\psi_i\gamma_i)=\qty(\sum\limits_{i=1}^k\gamma_i^*\psi_i\gamma_i)\circ\phi=\sum\limits_{i=1}^k\gamma_i^*\psi_i\circ\phi\gamma_i=\sum\limits_{i=1}^k\gamma_i^*\phi_{n_i}^*\qty(\psi_i)\gamma_i
\]
for every matrix convex combination of linear maps. Moreover, note that if $k\leq N$, then for every $\psi\in\UCP{\cR}{\mat{k}}$, the map $\phi_k^*\qty(\psi)=\psi\circ\phi$ is a composition of $k$ positive maps and thus makes a $k$-positive map into $\mat{k}$, which is completely positive (by Theorem \ref{thm:ChoiPos}). Assume that $\psi\in\UCP{\cR}{\mat{k}}$ is a matrix extreme point of $\UCP{\cR}$, and assume that $\phi_k^*\psi$ is a proper matrix convex combination of the form:
\[
	\phi^*\psi=\sum\limits_{i=1}^{\ell}\gamma_i^*\psi_i\gamma_i
\]
for $\psi_i\in\UCP{\cS}{\mat{k_i}},\gamma_i\in\mat{k_i}{k}$ for $i=1,\dots,\ell$ satisfying $\sum_{i=1}^{\ell}\gamma_i^*\gamma_i=\mathds{1}_k$, and that $k_i\leq k$ for all $i$. Because $\phi^{-1}$ is an $N$ positive map, we get that:
\[
	\psi=\qty(\phi^{-1})_k^*\phi_k^*\psi=\sum\limits_{i=1}^{\ell}\gamma_i\qty(\phi^{-1})_{k_i}^*\psi_i\gamma_i
\]
is a matrix convex combination in $\UCP{\cR}$ from levels below $k$. $\psi$ is a matrix extreme point, which means that $k_i=k$ and $\psi=u_i^*\qty(\phi^{-1})_{k_i}^*\psi_iu_i$ for some unitary $u_i$. But from this it follows that $\phi_k^*\psi=u_i^*\psi_iu_i$, which proves that $\phi_k^*\psi$ is matrix extreme. Similar arguments for ${\bm{\phi^{-1}}}^*$ shows that the map is indeed bijective.
\end{proof}
\begin{prop}
\label{prop:NordisMatAffines}
Let $\phi:\cS\to\cR$ be an $N$-order isomorphism. Then, $\bm{\phi}^*$ restricts to a matrix affine homeomorphism of $\overline{{\rm co}}\qty{\partial\UCP{\cR}{\mat{k}}}_{k=1}^N$ and $\overline{{\rm co}}\qty{\partial\UCP{\cS}{\mat{k}}}_{k=1}^N$.
\end{prop}
\begin{proof}
We already showed that the restriction of $\bm{\phi}^*$ to the matrix extreme points at levels below $N$ is a bijection. But because $\bm{\phi}$ is matrix affine with ${\bm{\phi^{-1}}}^*$ as an inverse, we get that it also restricts to a matrix affine invertible map between ${\rm co}\qty{\partial\UCP{\cR}{\mat{k}}}_{k=1}^N$ and ${\rm co}\qty{\UCP{\cS}{\mat{k}}}_{k=1}^N$. In order to extend this map to the closure, it suffices to prove that $\bm{\phi}^*_k:L(\cR,\mat{k})\to L(\cS,\mat{k})$ is continuous for all $k\leq N$. Indeed, since the topology is the weak* topology, we get that $\psi_n$ converges to $\psi$ in $L(\cR,\mat{k})$ if and only if $\psi_n(r)$ converges to $\psi(r)$ for all $r\in\cR$. But this means that for all $s\in\cS$, $\phi(s)\in \cR$ and thus $\phi_k^*(\psi_n)(s)=\psi_n(\phi(s))$ converges to $\psi(\phi(s))=\phi_k^*(\psi)(s)$ and thus $\phi_k^*\psi_n$ converges to $\phi_k^*\psi$, which means that $\phi_k^*$ is continuous. Finally, we get that a limit point of ${\rm co}\qty{\partial\UCP{\cR}{\mat{k}}}_{k=1}^N$ is mapped to a limit point of ${\rm co}\qty{\partial\UCP{\cS}{\mat{k}}}_{k=1}^N$, which completes the proof.
\end{proof}
We are ready to provide a proof for Theorem \ref{thm:NorderisComplete}.
\begin{proof}[Proof of Theorem \ref{thm:NorderisComplete}]
A complete order isomorphism is in particular an $N$-order isomorphism, so we will only prove the converse. Assuming $\phi$ is an $N$-order isomorphism, we use Proposition \ref{prop:NordisMatAffines}, to show that $\bm{\phi}^*:\UCP{\cR}\to\UCP{\cS}$ is a matrix affine homeomorphism (recall that both sets are $n$-generated for $n\leq N$). Therefore, we can now apply Proposition \ref{prop:mataffhomImpliesCOI} to get a complete order isomorphism of $\cS$ and $\cR$ given by $\delta_{\cR}^{-1}\circ\qty(\bm{\phi}^*)_*\circ\delta_{\cS}$. This completes the proof.
\end{proof}
It is worth noting, that a direct calculation which follows the last part of the proof, shows that:
\[
	\qty(\qty(\bm{\phi}^*)_*\circ\delta_{\cS}\qty(s))_k\qty(\psi)=\qty(\delta_{\cS}\qty(s))_k\circ \phi_k^*\qty(\psi)=\qty(\delta_{\cS}\qty(s))_k\qty(\psi\circ\phi)=\psi\qty(\phi\qty(s))=\qty(\delta_{\cR}\qty(\phi(s)))_k\qty(\psi),
\]
from which it holds that $\delta_{\cR}^{-1}\circ\qty(\bm{\phi}^*)_*\circ\delta_{\cS}=\phi$. Therefore, the $N$-order isomorphism $\phi$ is the complete order isomorphism.
\subsection{Subhomogeneous matrix ranges}
\begin{thm}
\label{thm:NordForRanges}
Let $\bs,\br$ be two $d$-tuples of operators such that $\cS(\bs)$ is $N_{\cS}$-Subhomogeneous and $\cS(\br)$ is $N_{\cR}$-Subhomogeneous, respectively. For $N=\max\{N_{\cS},N_{\cR}\}$, the following are equivalent:
\begin{enumerate}
\item $\bs$ and $\br$ are $N$-order equivalent.
\item $\bs$ and $\br$ are completely order equivalent.
\item $\W{N}{\bs}=\W{N}{\br}$.
\item $\W{\bs}=\W{\br}$.
\end{enumerate}
\end{thm}
\begin{proof}
$\qty(1)\Longleftrightarrow\qty(2)$ is a consequence of Theorem \ref{thm:NorderisComplete}. $\qty(1)\Longrightarrow\qty(3)$ and $\qty(4)\Longleftrightarrow\qty(2)$ follow from Theorem \ref{thm:DDSS}, which means we only need to prove $\qty(3)\Longrightarrow\qty(4)$. But note that by Proposition \ref{prop:NgeneratedisMinimal}, we have that:
\[
	\W{\bs}=\cW^{N\mbox{-min}}\qty(\W{\bs})={\rm co}\qty(\W{N}{\bs})={\rm co}\qty(\W{N}{\br})=\cW^{N\mbox{-min}}\qty(\W{\br})=\W{\br}
\]
which completes the proof.
\end{proof}
In general, choosing an $N$ lower then $\max\{N_{\cS},N_{\cR}\}$ might not be sufficient. Before giving a concrete example, we will require the following lemma.
\begin{lem}
\label{lem:numericalRange}
Let $\bs,\br$ be any two $d$-tuples. Then, $\W{1}{\bs}=\W{1}{\br}$ if and only if the tuples are $1$-order equivalent.
\end{lem}
\begin{proof}
By \ref{thm:DDSS}, the tuples being $1$-order equivalent implies that $\W{1}{\bs}=\W{1}{\br}$, which means we only have to prove the converse. We assume that the equality holds and show that the mapping $s_i\mapsto r_i$ extends to a well defined $1$-order isomorphism of the generated operator systems. Define $\phi:\cS(\bs)\to\cS(\br)$ by:
\[
	\phi\qty(a_0\mathds{1}_{\bs}+\sum_ia_is_i+b_is_i^*):=a_0\mathds{1}_{\br}+\sum_ia_ir_i+b_ir_i^*.
\]
We show that whenever the argument on the left hand side is positive, so is the image on the right hand side. Positive maps are bounded, which will also prove that this map is a well defined extension from ${\rm span}\left\{s_i\right\}$ to $\cS(\bs)$. Assuming positivity for the argument on the right hand side, we will prove that for any state $\psi:\cS(\br)\to\mathbb{C}$:
\[
	\psi\qty(a_0\mathds{1}_{\br}+\sum_ia_ir_i+b_ir_i^*)=a_0+\sum_ia_i\psi\qty(r_i)+b_i\psi\qty(r_i)^*\geq 0.
\]
Indeed, because $\psi:\cR\to\mathbb{C}=\mathbb{M}_1$ is positive, it is completely positive, which means:
\[
	\qty(\psi\qty(r_1),\dots,\psi\qty(r_d))\in\W{1}{\br}=\W{1}{\bs},
\]
and thus $\qty(\psi\qty(r_1),\dots,\psi\qty(r_d))$ is of the form $\qty(\rho\qty(s_1),\dots,\rho\qty(s_d))$ for some completely positive $\rho:\cS\to\mathbb{C}$. We can now conclude that:
\[
	a_0+\sum_ia_i\psi\qty(r_i)+b_i\psi\qty(r_i)^*=a_0+\sum_ia_i\rho\qty(s_i)+b_i\rho\qty(s_i)^*=\rho\qty(a_0\mathds{1}_{\bs}+\sum_ia_is_i+b_is_i^*)\geq 0,
\]
which is what we wanted to prove. This gives is positivity for $\phi$ and positivity of $\phi^{-1}$ follows the same argument as for $\phi$, which means $\phi$ is a $1$-order isomorphism.
\end{proof}
The next is an example of $1$-Subhomogeneous tuple and a $2$-Subhomogeneous tuple, which are $1$-order equivalent but not $2$-order equivalent, thus proving that in general, Theorem \ref{thm:NordForRanges} cannot be improved by choosing $N\leq\max\{N_{\cS},N_{\cR}\}$.
\begin{ex}
Let $\cH=L^2\qty(\bar{\mathbb{B}}_2)$ such that $\bar{\mathbb{B}}_2$ is the closed unit ball in $\mathbb{R}^2$, and let $M_{x_1},M_{x_2}$ be the multiplication by coordinate functions $M_{x_i}f\qty(x_1,x_2)=x_if\qty(x_1,x_2)$ for $i=1,2$. $\qty(M_{x_1},M_{x_2})$ is a pair of self-adjoint and commuting operators, which means by \cite[Corollary 4.4]{DDSS} (and also, the first paragraph of Section 3) that the matrix range of this set is the minimal matrix convex set which has $\sigma\qty(M_{x_1},M_{x_2})$ at its first level:
\[
	\W{M_{x_1},M_{x_2}}=\cW^{\mbox{min}}\qty(\sigma\qty(M_{x_1},M_{x_2})).
\]
Moreover, \cite[Theorem 2.7]{DDSS} tells us that $\W{1}{M_{x_1},M_{x_2}}={\rm conv}\left(\sigma\qty(M_{x_1},M_{x_2})\right)=\bar{B}_2$. We then consider another pair $\qty(F_1,F_2)\in B(\mathbb{C}^2)^2$ of matrices given by:
\[
	F_1=\mqty(1 & 0 \\ 0 & -1),\quad F_2=\mqty(0 & 1 \\ 1 & 0).
\]
It is easily verified that this is a pair of anti-commuting unitaries. Applying \cite[Corollary 5.9]{passer2018minimal}, we get that the mapping $M_{x_i}\mapsto F_i$ for $i=1,2$, extends to a unital and positive map from $\cS(M_{x_1},M_{x_2})$ to $\cS(F_1,F_2)$, and then, by Theorem \ref{thm:DDSS}:
\[
	\W{1}{F_1,F_2}\subset\W{1}{M_{x_1},M_{x_2}}.
\]
We can prove that the converse inclusion also holds. Note that for every $\xi=\qty(\xi_1,\xi_2)\in\mathbb{C}^2$ with $|\xi_1|^2+|\xi_2|^2=1$, we get that:
\[
	\qty(\xi^*F_1\xi,\xi^*F_2\xi)\in\W{1}{F_1,F_2}.
\]
Therefore, we have that for all $\theta\in[0,2\pi]$, we can choose $\xi=\qty(r_1e^{i\theta_1},r_2e^{i\theta_2})$ with:
\[
	r_1=\sqrt{\frac{1+\cos{(\theta)}}{2}},\quad r_2=\sqrt{\frac{1-\cos{(\theta)}}{2}},\quad \cos{(\theta_1-\theta_2)}={\rm sgn}\qty(\sin{(\theta)}),
\]
and get that:
\[
	\norm{\xi}=1,\quad \qty(\xi^*F_1\xi,\xi^*F_2\xi)=\qty(\cos{(\theta)},\sin{(\theta)}).
\]
This proves that $\W{1}{F_1,F_2}$ contains the unit circle in $\mathbb{R}^2$, and by convexity:
\[
	\W{1}{M_{x_1},M_{x_2}}=\bar{B}_2\subset\W{1}{F_1,F_2}
\]
which gives the equality. By Lemma \ref{lem:numericalRange}, we have that the pairs are $1$-order equivalent.\par
We now show that they cannot be $2$-order equivalent. Indeed, by Theorem \ref{thm:RationalMatrixRange} we have that $\W{F_1,F_2}$ is $2$-Subhomogeneous, while $\W{M_{x_1},M_{x_2}}$ is $1$-Subhomogeneous (by minimality). Therefore, a $2$-order isomorphism would imply an existence of a matrix extreme point in $\W{2}{M_{x_1},M_{x_2}}$ by Proposition \ref{prop:bijectionofMatExt} which is impossible.
\end{ex}
Another consequence of the fact that $N$-Subhomogeneous operator systems correspond to $N$-Subhomogeneous matrix range, is the following generalization of Theorem \ref{thm:ChoiPos} for Subhomogeneous operator systems.
\begin{thm}
\label{thm:NgeneratedChoiProperty}
Let $\bs,\br$ be two operator systems such that $\cS(\bs)$ is $N$-Subhomogeneous. If the mapping $r_i\mapsto s_i$ defines a unital $N$-positive map $\phi:\cS(\br)\to\cS(\bs)$, it is completely positive.
\end{thm}
\begin{proof}
By Proposition \ref{prop:NgeneratedisMinimal}, $\W{\bs}=\cW^{N\mbox{-min}}(\bs)$. For any UCP map $\psi:\cS(\bs)\to\mat{n}$ (with $n\leq N$), we have that $\psi\circ\phi:\cS(\br)\to \mat{n}$ is completely positive, which means that:
\[
	\W{n}{\bs}\subset\W{n}{\br}.
\]
From this inclusion, and the minimality of $\W{\bs}$:
\[
	\W{\bs}={\rm co}\qty(\W{N}{\bs})\subset \W{\br}.
\]
We can now use Theorem \ref{thm:DDSS} again to conclude that the mapping $r_i\mapsto s_i$ extends to a UCP map, which is precisely what we wanted to prove.
\end{proof}
\subsection{Matrix ranges generated by unitaries}
The case where $\cS,\cR$ are both generated by $d$-tuples $\bs,\br$ of unitaries is of particular interest. This is because unitaries are known to have some ``rigidity'' properties with respect to the algebraic structure of the algebras that they generate.\par
Recall that for an operator system $\cS$ and a $C^*$-algebra $\mathcal{B}=B(\cH)$, we say that the map $\phi\in\UCP{\cS}{\cB}$ has the \emph{unique extension property} if $\phi$ has a unique UCP extension to some $\Phi\in\UCP{C^*(\cS)}{\cB}$, and this map is a $*$-homomorphism. UCP maps which map unitary $d$-tuples to unitary $d$-tuples always have this property.
\begin{thm}
\label{thm:UnitariesUEP}
Let $\bs,\br$ be $d$-tuples of unitaries and assume that $s_i\mapsto r_i$ defines a UCP $\phi:\cS(\bs)\to\cS(\br)$. Then, this map has the unique extension property.
\end{thm}
\begin{proof}
We may assume without loss of generality that $C^*(\bs)$ and $C^*(\br)$ are concrete subalgebras of operators on Hilbert spaces $\cH,\cK$ respectively. In that case, we may apply Arveson's extension theorem (see \cite[Corollary 1.2.3]{Arveson69}) and find a UCP $\Phi:C^*(\bs)\to B(\cK)$ which extends $\phi$. Then, we may apply Stinespring's dilation theorem (see \cite[Theorem 1]{Stinespring}) to find a Hilbert space $\cK'$, a $*$-homomorphism $\pi:C^*(\bs)\to B(\cK')$ and an isometry $V:\cK\to\cK'$ such that for all $s\in C^*(\bs)$:
\[
	\Phi\qty(s)=V^*\pi(s)V.
\]
Since $V$ is an isometry, we may also identify $\cK$ with the closed subspace $V\cK\subset\cK'$. This way, we can decompose $\cK'=\cK\oplus\cK^{\perp}$ and write for all $a\in C^*(\bs)$:
\[
	\pi\qty(a)=\mqty(\Phi\qty(a) & \beta \\ \gamma & \delta).
\]
In particular, for all $i$, we have that $s_i$ and $\phi(s_i)=\Phi(s_i)$ are unitaries, which means that:
\begin{multline*}
	\pi\qty(s_i)^*\pi(s_i)=\mqty(\phi(s_i)^*\phi(s_i)+\gamma_i^*\gamma_i & \phi(s_i)^*\beta_i+\gamma_i^*\delta_i \\ \beta_i^*\phi(s_i)+\delta_i^*\gamma_i & \beta_i^*\beta_i+\delta_i^*\delta_i) \\=\mqty({\rm Id}_{\cK}+\gamma_i^*\gamma_i & \phi(s_i)^*\beta_i+\gamma_i^*\delta_i \\ \beta_i^*\phi(s_i)+\delta_i^*\gamma_i & \beta_i^*\beta_i+\delta_i^*\delta_i)=\mqty({\rm Id}_{\cK} & 0 \\ 0 & {\rm Id}_{\cK^{\perp}}),
\end{multline*}
so we can conclude that $\gamma_i^*\gamma_i=0$ (and similarly, $\beta_i^*\beta_i=0$), which implies that $\gamma_i=\beta_i=0$. Therefore:
\[
	\pi\qty(s_i)=\mqty(r_i & 0 \\ 0 & \delta_i)
\]
for all $i=1,\dots,d$, and since $s_i$ generates $C^*(\bs)$, we get that $\Phi$ is multiplicative, which makes it a $*$-homomorphism.
\end{proof}
\begin{cor}
\label{cor:COIForUnitariesIsIsomorphism}
For a $d$-tuples $\bs,\br$ of unitaries, the mapping $s_i\mapsto r_i$ defines a complete order isomorphism of $\cS(\bs)$ and $\cS(\br)$ if and only if it extends to a $*$-isomorphism of $C^*(\bs)$ and $C^*(\br)$. In particular, that $*$-isomorphism is the unique UCP extension of that map.
\end{cor}
For a tuple $\bs$, we say that $\bs$\emph{ dilates} to $\br$ if there are faithful $*$-representations $\pi:C^*(\bs)\to B(\cH),\rho:C^*(\br)\to B(\cK)$ and an isometry $V:\cH\to\cK$ such that:
\[
	\pi\qty(s_i)=V^*\rho(r_i)V
\]
for all $i=1,\dots,d$. Note that identifying $V\cH$ with $\cH$ as closed subspaces of $\cK$, we can also write the dilation as:
\[
	\pi\qty(s_i)=P_{\cH}\left.\rho(r_i)\right|_{\cH},
\]
Note that the dilation gives rise to a UCP map of $\cS(\br)$ to $\cS(\bs)$ which maps $r_i$ to $s_i$. The converse is also true, meaning that the existence of such a UCP map implies a dilation by using Arveson's extension theorem and Stinespring's dilation theorem. We can now give significant improvement of Theorem \ref{thm:NordForRanges} for tuples of unitaries.
\begin{thm}
\label{thm:equivalenceOfUnitaries}
Let $\bs,\br$ be two $d$-tuples of unitaries such that $\W{\bs}$ and $\W{\br}$ are $N_{\bs}$-Subhomogeneous and $N_{\br}$-Subhomogeneous, respectively. Setting $N=\max\{N_{\bs},N_{\br}\}$, the following are equivalent.
\begin{enumerate}
\item $\bs$ and $\br$ are $N$-order equivalent.
\item $\bs$ and $\br$ are completely order equivalent.
\item $\bs$ dilates to $\br$ and $\br$ dilates to $\bs$.
\item $\W{N}{\bs}=\W{N}{\br}$.
\item $\W{\bs}=\W{\br}$.
\item $\bs$ and $\br$ are $*$-isomorphic.
\end{enumerate}
\end{thm}
\begin{proof}
We already have $\qty(1)\Longleftrightarrow\qty(2)\Longleftrightarrow\qty(4)\Longleftrightarrow\qty(5)$. Note that $\qty(2)\Longrightarrow\qty(6)$ follows Corollary \ref{cor:COIForUnitariesIsIsomorphism}, and the preceeding remarks are essentially the proof of $\qty(3)\Longleftrightarrow\qty(6)$.
\end{proof}
\section{Operator systems related to the noncommutative tori}
In this section, we apply the tools developed so far to the theory of $q$-commuting unitaries and, more generally, $\Lambda$-commuting unitaries. Given a self-adjoint matrix $\Lambda=\qty(\lambda_{ij})_{i,j=1}^d$ with $\qty|\lambda_{ij}|=1$, we say that a $d$-tuple $u_1,\dots,u_d$ of unitaries are $\Lambda$\emph{-commuting} if $u_iu_j=\lambda_{ij}u_ju_i$. When $d=2$, $\Lambda$ is uniquely determined by a single complex number $q$ of modulus $1$, and in this case we say that $u,v$ are $q$-\emph{commuting }if $uv=qvu$.
\subsection{$q$-commuting unitaries}
When $q$ is of the form $q=e^{2\pi i\frac{k}{n}}$ for co-prime $k,n\in\mathbb{N}$, we will use the notation $q_{n,k}$ for convenience. The following proposition is folklore, and we provide a proof for completeness.
\begin{prop}
\label{prop:RationalRepresentations2}
Let $u,v$ be a pair of $q_{k,n}$-commuting unitaries, and let $\pi:C^*(u,v)\to B(\mathcal{H})$ be an irreducible representation. Then ${\rm dim}\,\mathcal{H}=n$ and there exists an orthonormal basis for $\cH$, such that $\pi\qty(u),\pi\qty(v)$ can be written in the form:
\[
	\pi\qty(u)=\lambda U,\quad \pi\qty(v)=\eta V,
\]
where $U,V$ are given by:
\begin{equation}
\label{eqn:cannonicalFormUV}
	U={\rm diag}\qty(1,q_{n,k},\dots,q_{n,k}^{n-1}),\quad V=\mqty(0 & 0 & \dots & 0 & 1 \\ 1 & 0 & \dots & 0 & 0  \\ 0 & 1 & \dots & 0 & 0 \\ \vdots & \ddots & \ddots & \ddots & \vdots \\ 0 & 0 & \ddots & 1 & 0).
\end{equation}
The pair $U,V$ will be referred as the \emph{standard representation} of $q_{n,k}$-commuting unitaries.
\end{prop}
\begin{proof}
Assume that $\pi:C^*(u,v)\to B(\cH)$ is an irreducible representation and define $\tilde{u}=\pi(u),\tilde{v}=\pi(v)$. For all $k,m\in\mathbb{N}$, we have that:
\[
	\tilde{u}^n\pi\qty(u^kv^m)=\pi\qty(u^{n+k}v^m)=q_{k,n}^{nm}\qty(u^kv^mu^n)=\pi\qty(u^kv^m)\tilde{u}^n,
\]
because $q_{k,n}^n=1$. The same calculation works for $\tilde{v}^n$ which shows that $\tilde{u}^n,\tilde{v}^n\in\pi\qty(C^*(u,v))'$. Therefore, there exists some $\xi,\zeta$ of modulus $1$ (because $\tilde{u},\tilde{v}$ are unitaries), such that $\tilde{u}^n=\xi {\rm Id}_{\cH}$ and $\tilde{v}^n=\zeta{\rm Id}_{\cH}$. Given some non-zero vector $h\in\cH$, we have
\[
	\cK={\rm span}\qty{\tilde{u}^{\ell}\tilde{v}^mh\middle|\ell,m\in\qty{-n+1,\dots,n-1}}\subset\cH
\]
is invariant of $\pi(C^*(u,v))$, and by irreducibility, $\cK=\cH$, from which it holds that $\cH$ is finite dimensional. Now, we can use the fact that $\tilde{u}$ is unitary, to find an eigenvector $\tilde{h}$ for $u$ for some eigenvalue $\lambda$ of modulus $1$, so that $\tilde{u}\tilde{h}=\lambda\tilde{h}$. But then, for all $\ell\in\mathbb{N}$:
\[
	\tilde{u}\tilde{v}^{\ell}\tilde{h}=q_{k,n}^{\ell}\tilde{v}^{\ell}\tilde{u}\tilde{h}=\lambda q_{k,n}^{\ell}\tilde{v}^{\ell}\tilde{h},
\]
so that $\qty{\tilde{h},\tilde{v}\tilde{h},\dots,\tilde{v}^{n-1}\tilde{h}}$ is a basis consisting of eigenvectors of $\tilde{u}$ for $n$ distinct eigenvalues. But again, we have that their span is an invariant subspace for $\pi(C^*(u,v))$, which means that $\cH={\rm span}\qty{\tilde{h},\dots,\tilde{v}^{n-1}\tilde{h}}$.  Finally, choose some $\eta$ such that $\eta^n=\zeta^{-1}$. A direct calculation shows that with respect to the basis $\qty{\eta^{i-1}\tilde{v}^{i-1}\tilde{h}}_{i=1}^n$, $\tilde{u}$ and $\tilde{v}$ take the desired form.
\end{proof}
Since any boundary representation is irreducible by definition, we now have the following corollary.
\begin{cor}
\label{cor:RationalAbex}
Let $\cS=\cS(u,v)$ be an operator system generated by $q_{k,n}$-commuting unitaries. Then, ${\rm Abex}\qty(\UCP{\cS})$ is not empty, and is contained in $\partial\UCP{\cS}{\mat{n}}$.
\end{cor}
\begin{proof}
First, we know by Theorem \ref{thm:absoluteIsBoundary} that ${\rm Abex}\qty(\UCP{\cS})$ are precisely the restrictions of boundary representations for $C^*(u,v)$ with respect to $\cS(u,v)$. Boundary representations are irreducible so in our case, Proposition \ref{prop:RationalRepresentations2} shows that the boundary representation are mappings into $\mat{n}$ so that ${\rm Abex}\qty(\UCP{\cS})\subset\UCP{\cS}{\mat{n}}$. By \cite[Theorem 3.4]{davidson2015choquet}, we know that any operator system admits a boundary representation, which also shows that ${\rm Abex}\qty(\UCP{\cS})$ is non-empty.
\end{proof}
Using this corollary and our knowledge of matrix extreme points, we can give a characterization of $\UCP{\cS}$ of an operator system generated by $q_{k,n}$-commuting unitaries.
\begin{thm}
\label{thm:RationalMatrixRange}
Let $\cS=\cS\qty(u,v)$ be an operator system such that $u,v$ are $q_{k,n}$-commuting unitaries. Then:
\begin{enumerate}
\item $\UCP{\cS}$ is $n$-Subhomogeneous.
\item $\partial\UCP{\cS}{\mat{n}}={\rm Abex}\qty(\UCP{\cS})$, and this is the set of all restrictions of boundary representations for $\cS$.
\item For all $1\leq k< n$, $\partial\UCP{\cS}{\mat{k}}$ is non-empty, and any $\varphi\in\partial\UCP{\cS}{\mat{k}}$ can be written in one of two forms:
\begin{itemize}
\item $\varphi=\gamma^*\psi\gamma$ for $\psi\in\partial\UCP{\cS}{\mat{k+1}}$.
\item $\varphi=\gamma^*\psi\gamma$ for $\psi\in{\rm Abex}\qty(\UCP{\cS})$.
\end{itemize}
\end{enumerate}
\end{thm}
\begin{proof}
Let $\varphi$ be a matrix extreme point in $\UCP{\cS}$. By Theorem \ref{thm:matExtremeArePure}, we know that $\varphi\in\UCP{\cS}{\mat{k}}$ is a pure UCP map, so by Theorem \ref{thm:pureDilatesToBoundary}, it is of the form $\varphi=\gamma^*\psi\gamma$ for an isometry $\gamma$ and a restriction of a boundary representation $\psi=\left.\pi\right|_{\cS}$. Since $\pi$ is a boundary representation, it is of dimension $n$ by Proposition \ref{prop:RationalRepresentations2}, which means that $\psi\in\UCP{\cS}{\mat{n}}$. From this it follows that $k\leq n$. Proposition \ref{prop:RationalRepresentations2} also shows that ${\rm Abex}\qty(\UCP{\cS})\subset\partial\UCP{\cS}{\mat{n}}$ is non-empty, which proves that $\UCP{\cS}$ is $n$-Subhomogeneous. Lastly, if $\varphi\in\partial\UCP{\cS}{\mat{n}}$, it is either in ${\rm abex}\qty(\UCP{\cS})$ or it dilates to some matrix extreme point $\psi\in\partial\UCP{\cS}{\mat{n+1}}$ by \cite[Lemma 2.3]{davidson2015choquet}. Since $\partial\UCP{\cS}{\mat{n+1}}$ is empty, $\varphi$ has to be in ${\rm abex}\qty(\UCP{\cS})$, and this completes the proof of the first and second parts of the theorem. Lastly, note that $\UCP{\cS}{\mathbb{C}}$ is a compact and convex set, so by the Krein-Milman theorem, it has a classical extreme point, which has to be matrix extreme point in a trivial sense. By \cite[Lemma 2.3]{davidson2015choquet}, it dilates to a matrix extreme point in $\partial\UCP{\cS}{\mat{2}}$, which is not an absolute matrix extreme point unless $n=2$ (in which case we are done). When $n>2$, we can use the same lemma to dilate this point to other matrix extreme points in $\partial\UCP{\cS}{\mat{k}}$ for all $2\leq k\leq n$, and then at level $n$, the point will be an absolute matrix extreme point and will not dilate to another point at a higher level. This proves that all levels up to level $n$ are non-empty in $\partial\UCP{\cS}$, and that all points below level $n$ are compressions of matrix  extreme points from one level higher, which completes the proof.
\end{proof}
We now move to consider the case where the angle is an irrational multiple of $\pi$, that is, $q=e^{2\pi i\theta}$ with $\theta\in\qty[0,1]\setminus\mathbb{Q}$.
\begin{prop}
\label{prop:irrationalRepresentations2}
Let $u,v$ be a pair of $q$-commuting unitaries where $q$ is irrational, and let $\pi:C^*(u,v)\to B(\cH)$ be an irreducible representation. Then $\cH$ is infinite dimensional.
\end{prop}
\begin{proof}
Assume by contradiction that $\cH$ is finite dimensional. In that case, $\pi(u)$ is a unitary acting on a finite dimensional vector space, which means that it has a non-zero eigenvector $h\in\cH$ for some non-zero eigenvalue $\lambda\in\mathbb{T}$. In that case, $\pi(v)^kh$ is a non-zero vector satisfying:
\[
	\pi(u)\pi(v)^kh=q^k\pi(v)^k\pi(u)h=\lambda q^k\pi(v)^kh
\]
which means that $\qty{\pi(v)^kh}_{k=0}^{\infty}$ is an infinite set of eigenvectors of distinct eigenvalues, and therefore, they are linearly independent which contradicts the fact that $\cH$ is finite dimensional.
\end{proof}
This provides us with enough information to describe the matrix state space of $\UCP{\cS}$ in the irrational case.
\begin{thm}
\label{thm:IrrationalMatrixRange}
Let $\cS=\cS(u,v)$ be an operator system such that $u,v$ are $q$-commuting unitaries for irrational $q$. Then:
\begin{enumerate}
\item $\UCP{\cS}$ is not finitely generated.
\item ${\rm abex}\qty(\UCP{\cS})=\emptyset$.
\item $\partial\UCP{\cS}{\mat{k}}$ is non-empty for all $k\in\mathbb{N}$, and any $\varphi\in\partial\UCP{\cS}{\mat{k}}$ can be written in the two forms:
\begin{itemize}
\item $\varphi=\gamma^*\psi\gamma$ for an isometry $\gamma$ and $\psi\in\partial\UCP{\cS}{\mat{k+1}}$.
\item $\varphi=V^*\left.\pi\right|_{\cS}V$ for an isometry $V$ and $\pi$ a boundary representation for $\cS$.
\end{itemize}
\end{enumerate}
\end{thm}
\begin{proof}
$\UCP{\cS}{\mathbb{C}}$ is a compact convex set in a locally convex topological vector space. By the Krein-Milman theorem, there is a classical extreme point in this set, which has to be a matrix extreme point, so $\partial\UCP{\cS}{\mathbb{C}}$ is non-empty. For such a matrix extreme point $\varphi$, we know that it is a pure point by Theorem \ref{thm:matExtremeArePure}, but it is not maximal, because then it would be a restriction of a boundary representation by \cite[Lemma 2.1]{davidson2015choquet} which contradicts Proposition \ref{prop:irrationalRepresentations2}. We may now use \cite[Lemma 2.3]{davidson2015choquet}, and dilate $\varphi$ to a matrix extreme point at level $2$, so that $\partial\UCP{\cS}{\mat{2}}\neq\emptyset$. But we can now continue by repeating this process and dilate a matrix extreme point in $\partial\UCP{\cS}{\mat{n}}$ to a matrix extreme point in $\partial\UCP{\cS}{\mat{n+1}}$, because this process will reach a boundary representation in a finite amount of steps of this form. This shows that there are matrix extreme points at every level of the set, which proves $(1)$ and proves the first part of $(3)$. $(2)$ is an immediate consequence of Proposition \ref{prop:irrationalRepresentations2}, and Theorem \ref{thm:absoluteIsBoundary}. Finally, the second part of $(3)$ is the work of Davidson and Kennedy (see \cite[Theorem 2.4]{davidson2015choquet}).
\end{proof}
By using this characterization, we show that rational $q_{k,n}$-commuting unitaries can be detected through their $n$-order structure.
\begin{thm}
\label{thm:qSeperation}
Let $u,v$ be a pair of $q_{k,n}$-commuting unitaries, and let $\tilde{u},\tilde{v}$ be a pair of $q$-commuting unitaries ($q$ might be irrational). Assume that the mapping $\tilde{u}\mapsto u,\tilde{v}\mapsto v$ defines a unital and $n$-positive map. Then, $q=q_{k,n}$, and this map is completely positive. If this mapping is an $n$-order isomorphism, it extends to a $*$-isomorphism of $C^*(u,v)$ and $C^*(\tilde{u},\tilde{v})$.
\end{thm}
\begin{proof}
First note that by Theorem \ref{thm:RationalMatrixRange}, $\UCP{\cS(u,v)}$ is $n$-generated. By Theorem \ref{thm:NgeneratedChoiProperty}, the unital and $n$-positive map of $\cS(\tilde{u},\tilde{v})$ to $\cS(u,v)$ which is defined by $\phi(\tilde{u})=u,\phi(\tilde{v})=v$ is completely positive. Therefore, we invoke Theorem \ref{thm:UnitariesUEP} and get that $\phi$ has the unique extension property, and thus extends to a $*$-homomorphism of $C^*(\tilde{u},\tilde{v})$ to $C^*(u,v)$, which we denote as $\pi$. We now have that $uv=q_{k,n}vu$, but also:
\[
	uv=\pi(\tilde{u}\tilde{v})=q\pi\qty(\tilde{v}\tilde{u})=qvu,
\]
from which it follows that $q=q_{k,n}$. In the case where $\phi$ defines an $n$-order isomorphism of the corresponding operator systems, we also have an extension of $\phi$ to a $*$-isomorphism of the generated $C^*$-algebras.
\end{proof}
\subsection{$\Lambda$-commuting $d$-tuples}
Theorem \ref{thm:RationalMatrixRange} gives us a complete description of the boundary representations of $q$-commuting unitaries. But the proof relies, in part, on the fact we are working with pairs of $q$-commuting unitaries. Characterizing the boundary representations for more general $\Lambda$-commuting unitaries is harder. Even the dimensions of those representations are unknown in general (although there are some results, such as \cite[Proposition B.1]{gerhold2021dilations}). But when $\Lambda=\qty(\lambda_{ij})_{i,j=1}^d$ is a self-adjoint complex matrix with $\qty|\lambda_{ij}|=1$ and $\lambda_{ij}=e^{2\pi i\frac{k_{ij}}{n_{ij}}}$ are rational, we can still give some bound on the dimension of the irreducible representations.
\begin{prop}
\label{prop:RationalLambdaReps}
Let $\bs$ be a $d$-tuple of $\Lambda$-commuting unitaries for some rational $\Lambda$. Let $N={\rm lcm}\qty{n_{ij}}$. If $\pi:C^*(\bs)\to B(\cH)$ is irreducible, then $\cH$ is finite dimensional, and ${\rm dim}\,\cH\leq N^d$.
\end{prop}
\begin{proof}
Similar to the $q$-commuting case, we set $U_i=\pi\qty(s_i)$, and note that for all $r_1,\dots,r_d\in\mathbb{Z}$:
\[
	U_i^N\prod\limits_{j=1}^dU_j^{r_j}=\qty(\prod\limits_{\substack{j=1 \\ j\neq i}}^d\lambda_{ij}^{Nr_j})\qty(\prod\limits_{j=1}^dU_j^{r_j})U_i^N=\qty(\prod\limits_{j=1}^dU_j^{r_j})U_i^N
\]
which means that $U_i^N\in\pi\qty(C^*(\bs))'$ for $i=1,\dots,d$. By irreducibility, there exists some scalars $\eta_1,\dots,\eta_d$ of modulus $1$, such that $U_i^N=\eta_i{\rm Id}_{\mathcal{H}}$ for $i=1,\dots,d$. Picking some non-zero vector $\xi\in\cH$, we have that:
\[
	\cK={\rm span}\qty{\prod\limits_{j=1}^dU_j^{r_j}\xi\middle|r_j\in\qty{0,\dots,N-1}, j=1,\dots,d},
\]
from which irreducibility implies $\cK=\cH$, and $\cH$ is finite dimensional. Since the space is spanned by at most $N^d$ vectors, we also have the bound on the dimension.
\end{proof}
Note that the example in \cite[Proposition B.1]{gerhold2021dilations} shows a $3$-tuple of unitaries which are $\Lambda$-commuting for $\lambda_{i,i+1}=e^{2\pi i\frac{m}{n}}$, but are irreducibly represented on an $n$-dimensional Hilbert space, a dimension smaller then the estimate, $n^3$. Although we do not have an exact knowledge of the dimensions of irreducible representations of such tuples, we still have a bound. We can thus show that they are all Subhomogeneous.
\begin{prop}
\label{prop:RationalLambdaUCP}
Let $\Lambda=\qty(\lambda_{ij})_{i,j=1}^d$ be rational and set $\bs^{\Lambda}$ to be the universal $d$-tuple of $\Lambda$-commuting unitaries. Then, $\UCP{\cS\qty(\bs^{\Lambda})}$ is $N(\Lambda)$-Subhomogeneous for some $N(\Lambda)\leq N^d$, and whenever $\br$ is a $d$-tuple of $\Lambda$-commuting unitaries, then $\UCP{\cS\qty(\br)}$ is $n$-Subhomogeneous for some $n\leq N(\Lambda)$.
\end{prop}
\begin{proof}
First, note that by Proposition \ref{prop:RationalLambdaReps}, all boundary representations of $\cS(\bs^{\Lambda})$ are of dimension lower or equal to $N^d$. Let $N(\Lambda)$ be the largest integer for which there exits a boundary representation on an $N(\Lambda)$ dimensional Hilbert space. By combining Theorem \ref{thm:absoluteIsBoundary}, Theorem \ref{thm:matExtremeArePure}  and \cite[Theorem 2.4]{davidson2015choquet}, we get that all absolute matrix extreme points are in $\UCP{\cS(\bs^{\Lambda})}{k}$ with $k\leq N(\Lambda)$, and that any other matrix extreme points are compressions of such points, which means that all extreme points lie in $\UCP{\cS(\bs^{\Lambda})}{\mat{k}}$ for $k\leq N(\Lambda)$. Because there is an absolute matrix extreme point in $\UCP{\cS(\bs^{\Lambda})}{\mat{N(\Lambda)}}$, we have that $\UCP{\cS\qty(\bs^{\Lambda})}$ is $N(\Lambda)$-generated. Next, let $\br$ be a $d$-tuple of $\Lambda$-commuting unitaries. By the universal property of $\cS(\bs^{\Lambda})$, there exists a $*$-homomorphism $\rho:C^*(\bs^{\Lambda})\to C^*(\br)$ such that $\rho(s^{\Lambda}_i)=r_i$. Therefore, if $\varphi\in{\rm abex}\qty(\UCP{\cS(\br)})$, it is a restriction of a boundary representation $\pi:C^*(\br)\to B(\cH)$, which means that $\pi\circ\rho:C^*(\bs^{\Lambda})\to B(\cH)$ is a boundary representation, and thus an absolute matrix extreme point in $\UCP{\cS(\bs^{\Lambda})}$. Therefore, ${\rm dim}\,\cH\leq N(\Lambda)$. This means that the largest integer $n$ for which a boundary representation for $\cS(\br)$ exists, must satisfy $n\leq N(\Lambda)$, we get that $\UCP{\cS(\br)}$ is $n$-generated for some $n\leq N(\Lambda)$.
\end{proof}
We finish this part by extending Theorem \ref{thm:qSeperation} and Theorem \ref{thm:NordForRanges} for the case of $\Lambda$-commuting unitaries.
\begin{thm}
\label{thm:LambdaNposIsUCP}
Let $\Lambda$ be rational and let $\bs$ be a $d$-tuple of $\Lambda$-commuting unitaries. If $\br$ is a $d$-tuple of $\tilde{\Lambda}$-commuting unitaries. If the mapping $r_i\mapsto s_i$ extends to a unital $N(\Lambda)$-positive map $\phi:\cS(\br)\to\cS(\bs)$, then $\Lambda=\tilde{\Lambda}$. In particular, if the mapping is an $N(\Lambda)$-order isomorphism, the tuples are completely order equivalent and generate isomorphic $C^*$-algebras.
\end{thm}
\begin{proof}
Under the assumptions, we have that $\UCP{\cS(\bs)}$ is $n$-generated for some $n\leq N(\Lambda)$. By Theorem \ref{thm:NgeneratedChoiProperty}, we have that the unital $N(\Lambda)$-positive map which maps $r_i$ to $s_i$ for $i=1,\dots,d$, is completely positive. Therefore, we use Theorem \ref{thm:UnitariesUEP} to conclude that this mapping has the UEP, and thus extends to a $*$-homomorphism of the generated $C^*$-algebras. Since this map is multiplicative, we use similar calculations as the $d=2$ case, to get that $\Lambda=\tilde{\Lambda}$. If the mapping defines an $N(\Lambda)$-order isomorphism, we now have an inverse which is $N(\Lambda)$-positive, and get that this map extends to a $*$-isomorphism of the $C^*$-algebras.
\end{proof}
And finally:
\begin{thm}
\label{thm:LambdaEquivalence}
Let $\Lambda$ be rational and $\bs$ be a $d$-tuple of $\Lambda$-commuting unitaries. For a $d$-tuple $\br$ of $\tilde{\Lambda}$-commuting unitaries, the following are equivalent:
\begin{enumerate}
\item $\bs$ and $\br$ are $N(\Lambda)$-order equivalent.
\item $\bs$ and $\br$ are completely order equivalent.
\item $\bs$ dilates to $\br$ and $\br$ dilates to $\bs$.
\item $\W{N(\Lambda)}{\bs}=\W{N(\Lambda)}{\br}$.
\item $\W{\bs}=\W{\br}$.
\item $\bs$ and $\br$ are $*$-isomorphic.
\end{enumerate}
\end{thm}
\begin{proof}
$\qty(1)\Longleftrightarrow\qty(2)$ follows Theorem \ref{thm:LambdaNposIsUCP}. $\qty(2)\Longleftrightarrow\qty(3)$ follows the discussion following Corollary \ref{cor:COIForUnitariesIsIsomorphism}. $\qty(2)\Longleftrightarrow\qty(5)$. For $\qty(5)\Longrightarrow\qty(4)$, note that we only have to prove $\qty(4)\Longrightarrow\qty(5)$. Because $\W{\bs}$ is $n$-generated for some $n\leq N(\Lambda)$, by Proposition \ref{prop:RationalLambdaUCP}. Therefore, we invoke Proposition \ref{prop:NgeneratedisMinimal} and conclude that $\W{\bs}$ is $N(\Lambda)$-minimal, which means that $\W{\bs}\subset \W{\br}$ and by Theorem \ref{thm:DDSS}, there exists a UCP map of the corresponding operator systems, mapping $r_i$ to $s_i$. By Theorem \ref{thm:LambdaNposIsUCP},  this map extends to a $*$-homomorphism of the generated $C^*$-algebras, from which it follows that $\Lambda=\tilde{\Lambda}$. Now, we can repeat those arguments for $\W{\br}$, and get $\qty(5)$. Finally, $\qty(2)\Longleftrightarrow\qty(6)$ also follows Theorem \ref{thm:LambdaNposIsUCP}.
\end{proof}
\section{Classification up to unital isometries}
In this section we try to investigate whether the conditions of Theorem \ref{thm:qSeperation} are strict. Namely, does there exist a pair of $q_{k,n}$ commuting unitaries which is $m$-order equivalent (for $m<n$) to some $q$-commuting pair, but the two pairs are not completely order equivalent?\par 
More specifically, we move to consider unital isometries. Because unital isometries are positive, they are an example of $1$-order isomorphism (see \cite[Proposition 2.11]{paulsen}). The following theorem provides a relation between the existence of unital isometries to certain dilation properties of $q$-commuting unitaries.
\begin{thm}
\cite[Theorem 6.4]{gerhold2022dilations} Let $\theta,\theta'\in\qty[0,1],q=e^{2\pi i\theta},q'=e^{2\pi i\theta'}$ and $\gamma=\theta-\theta'$. The smallest constant $c_{\theta,\theta'}$ such that every pair of $q$-commuting unitaries can be dilated to $c_{\theta,\theta'}$ times a pair of $q'$-commuting unitaries is given by:
\[
	c_{\theta,\theta'}=c_{\gamma}=\frac{4}{\norm{u_{\gamma}+u_{\gamma}^*+v_{\gamma}+v_{\gamma}^*}},
\]
where $u_{\gamma},v_{\gamma}$ are the universal pair of $e^{i\gamma}$-commuting unitaries.
\end{thm}
Setting $c_{\theta}:=c_{\theta,0}$, we get from the theorem that the existence of a unital isometry between $(u_{\theta},v_{\theta})$ and $(u_{\theta'},v_{\theta'})$ (where the two pairs are universal), implies that $c_{\theta}=c_{\theta'}$. We do not know if the converse is true in general, but we can deduce that whenever such an isometry cannot exist, the corresponding dilation constants must be different. The following corollary now follows easily from our now set of tools.
\begin{cor}
For $\theta\in\qty[0,1]$, $c_{\theta}=1$ if and only if $\theta=0$ or $\theta=1$.
\end{cor}
\begin{proof}
If $\theta=0$ or $\theta=1$, there is nothing to prove, because every pair of commuting unitaries is a dilation of itself. On the other hand, if a pair of $q$-commuting unitaries dilate to a pair of commuting unitaries (with $q'=1$), we can invoke Theorem \ref{thm:qSeperation} (because a the dilation implies a completely positive map which maps the commuting unitaries to the $q$-commuting unitaries) and get that $q'=q$ so that $\theta=0$ or $\theta=1$. 
\end{proof}
When $\theta\neq 0,1$, there are still some cases of two different angles which correspond to the same dilation constant. For example, for all $\theta\in\qty[0,1]$, it is known that $c_{\theta}=c_{1-\theta}$. In the following proposition we provide another proof for this fact, and we also give an example of a large family of $1$-order isomorphic but not completely order isomorphic $q$-commuting pairs.
\begin{prop}
Let $\theta\in\qty[0,2\pi]$. Then, the mapping $\theta\mapsto u_{1-\theta},v_{\theta}\mapsto v_{1-\theta}$ extends to a unital isometry of the operator systems. Moreover, $c_{\theta}=c_{1-\theta}$.
\end{prop}
Before giving the proof, note that for all $n\in\mathbb{N}$, we have that the universal pair of $q_{1,n}$-commuting and the universal pair of $q_{(n-1),n}$-commuting unitaries are $1$-order equivalent but not completely order equivalent. This gives a family of examples for $n$-Subhomogeneous operator systems which are $1$-order isomorphic but not $n$-order isomorphic (and in particular, not completely order isomorphic). We now move to the proof of the proposition.
\begin{proof}
Assume first that $\theta=\frac{k}{n}$ for coprime $k,n\in\mathbb{N}$ and set $q=e^{2\pi i\theta},q'=e^{2\pi i(1-\theta)}$. Considering $U,V$ the standard representation as in (\ref{eqn:cannonicalFormUV}), we see that if $UV=qVU$, then:
\[
	U^tV^t=\qty(VU)^t=\bar{q}\qty(UV)^t=q'V^tU^t,
\]
which means that $U^t,V^t$ is a pair of $q'$-commuting unitaries. Moreover, the mapping $\phi:\mat{n}\to\mat{n}$ given by $\phi(a)=a^t$ is a unital isometry, which means that its restriction of $\cS(U,V)$ is a unital isometry. Next, we let $U_{\theta},V_{\theta}$ be the universal pair of $q$-commuting unitaries, which is given by:
\[
U_{\theta}=\bigoplus\limits_{\alpha,\beta\in\mathbb{T}}U_{\theta,\alpha,\beta},\quad V_{\theta}=\bigoplus\limits_{\alpha,\beta\in\mathbb{T}}V_{\theta,\alpha,\beta}
\]
where $U_{\theta,\alpha,\beta}=\alpha U$ and $V_{\theta,\alpha,\beta}=\beta V$. In that case, the mapping $U_{\theta}\mapsto U_{\theta}$ and:
\[
	V_{\theta}\mapsto\bigoplus\limits_{\alpha,\beta\in\mathbb{T}}V_{\theta,\alpha,\beta}^t
\]
is a unital isometry from $\cS(U_{\theta},V_{\theta})$ to $\cS(U_{\theta},V_{\theta}^t)$. But $U_{\theta},V_{\theta}^t$ are a universal pair of $\bar{q}=q'$ commuting unitaries, which completes the proof for the rational case.\par Assuming that $\theta$ is irrational, we consider the following pair of infinite matrices:
\[	u=\qty(q^n\delta_{n,m})_{n,m\in\mathbb{Z}},\quad v=\qty(\delta_{n,m+1})_{n,m\in\mathbb{Z}}
\]
which act on $\ell_2\qty(\mathbb{Z})$. Those are clearly $q$-commuting unitaries, and they form a universal pair (by \cite[Theorem 1.9 and 1.10]{boca2001rotation}). Working similarly to the rational case, we have that the transpose map $u\mapsto u^t=u$ and $v\mapsto v^t$, defines a unital isometry of the generated operator systems, where the second operator system is generated by $u,v^t$, which is a universal pair of $q'=\bar{q}$-commuting unitaries. This completes the proof.
\end{proof}
\bibliographystyle{plain}
\bibliography{ref}
\end{document}

%% file: macros.tex
\newcommand{\cA}{\mathcal{A}}
\newcommand{\cB}{\mathcal{B}}

\newcommand{\cR}{\mathcal{R}}
\newcommand{\cS}{\mathcal{S}}
\newcommand{\cW}{\mathcal{W}}

\newcommand{\cH}{\mathcal{H}}
\newcommand{\cK}{\mathcal{K}}

\newcommand{\br}{\mathbf{r}}
\newcommand{\bs}{\mathbf{s}}

\NewDocumentCommand{\mat}{mg}{%
  \IfValueTF{#2}{%
    \mathbb{M}_{#1,#2}
  }{%
    \mathbb{M}_{#1}%
  }%
}
\NewDocumentCommand{\UCP}{mg}{%
  \IfValueTF{#2}{%
    {\rm UCP}\left( #1,#2 \right)
  }{%
    {\rm UCP}\left(#1\right)
  }%
}
\NewDocumentCommand{\W}{mg}{%
  \IfValueTF{#2}{%
    \mathcal{W}_{#1}\left(#2\right)
  }{%
    \mathcal{W}\left(#1\right)
  }%
}